\documentclass[times,11pt]{elsarticle}
\usepackage{amsmath,amsfonts,amssymb,amsthm}
\usepackage{verbatim}
\usepackage{enumerate}
\usepackage{url}
\usepackage{tikz}
\usepackage[margin=3cm]{geometry}
\newtheorem{theorem}{Theorem}

\newtheorem{lemma}{Lemma}
\newtheorem{corollary}{Corollary}
\newtheorem{definition}{Definition}
\newtheorem{remark}{Remark}
\newtheorem{example}{Example}
\newcommand{\F}{\mathbb F}

\newcommand{\N}{\mathbb N}

\newcommand{\Char}{\mathrm{char}}
\newcommand{\GL}{\mathrm{GL}}

\newcommand{\Tr}{\mathrm{Tr}}
\newcommand{\ord}{\mathrm{ord}}

\begin{document}

\begin{frontmatter}



\title{The action of $\GL_2(\F_q)$ on irreducible polynomials over $\F_q$}

\author{Lucas Reis}
\ead{lucasreismat@gmail.com}

\address{Departamento de Matem\'{a}tica, Universidade Federal de Minas Gerais, UFMG, Belo Horizonte, MG, 30123-970, Brazil}


\begin{abstract}
Let $\F_q$ be the finite field with $q$ elements, $p=\Char  \F_q$. The group $\GL_2(\F_q)$ acts naturally in the set of irreducible polynomials over $\F_q$ of degree at least $2$. In this paper we are interested in the characterization and number of the irreducible polynomials that are fixed by the elements of a subgroup $H$ of $\GL_2(\F_q)$. We make a complete characterization of the fixed polynomials in the case when $H$ has only elements of the form $\left(\begin{matrix}
1&b\\
0&1
\end{matrix}\right)$, corresponding to translations $x\mapsto x+b$ and, as a consequence, the case when $H$ is a $p-$subgroup of $\GL_2(\F_q)$.
This paper also contains alternative solutions for the cases when $H$ is generated by an element of the form $\left(\begin{matrix}
a&0\\
0&1
\end{matrix}\right)$,
 obtained by Garefalakis (2010) \cite{gfq} and $H=\rm{PGL}$$_2(\F_q)$, obtained by Stichtenoth and Topuzoglu (2011) \cite{general}.
\end{abstract}

\end{frontmatter}





\section{Introduction}
Let $\F_q$ be the finite field with $q$ elements, $p=\Char \F_q$. As it was noticed in \cite{gfq} and \cite{general}, there is a natural action of the group $\GL_2(\F_q)$ on the set $I$ of irreducible polynomials of degree at leat $2$ in $\F_q[x]$. Namely, given
$A=
\left(\begin{matrix}
a&b\\
c&d
\end{matrix}\right)
$
in $\GL_2(\F_q)$ and $f(x)\in I$ of degree $n$ we can define $$A\circ f=(cx+d)^nf\left( \frac{ax+b}{cx+d}\right).$$
It can be verified that $f$ and $A\circ f$ have the same degree and $A\circ( B\circ f)=(AB)\circ f$ for all $A, B\in \GL_2(\F_q)$. From this definition, two interesting theoretical questions arise:
\begin{enumerate}[a)]
\item Given a subgroup $H$ of  $\GL_2(\F_q)$, which elements $f\in I$ are fixed by $H$, i.e., $A\circ f=f$ for all $A\in H$?
\item How many fixed elements exist?
\end{enumerate}

In \cite{general}, the authors gave a complete characterization of the elements $f\in I$ that are fixed by $H=\langle A\rangle$, where $A$ is any element of $\GL_2(\F_q)$. Earlier, the same characterization was given in \cite{gfq} for the special cases
$A=\left(\begin{matrix}
a&0\\
0&1
\end{matrix}\right)$ or $\left(\begin{matrix}
1&b\\
0&1
\end{matrix}\right)$, corresponding to the homothety $x\mapsto ax$ and the translation $x\mapsto x+b$, respectively, including enumeration formulas. Combining those results, in \cite{gfq} the author also obtained enumeration formulas for the case when $A$ is an upper triangular matrix in $\GL_2(\F_q)$.

In this paper we are also going to consider only translations and homotheties but with an alternative characterization of the invariant polynomials. Section 2 is devoted to extend the results of Garefalakis \cite{gfq} for translations, considering now any family of translations $\{x\mapsto x+s; s\in S\}$ where $S\subset \F_q$. Section 3 includes alternative proofs of the enumeration formula in \cite{gfq} for homotheties and the characterization of all irreducible polynomials that are invariant under the action of the whole group $\rm{PGL}$$_2(\F_q)$, [\cite{general}, Proposition 4.8]. In Section 4, using the canonical conjugation in $\GL_2(\F_q)$ and the results of Section 2,  we count the irreducible polynomials that are fixed by any given $p-$subgroup $H$ of $\GL_2(\F_q)$.

\section{S-Translation Invariant Polynomials}

Throughout this section $\F_q$ denotes the finite field with $q$ elements, where $q=p^k$ is a prime power. If $S\subset \F_q$, we say that $f(x)\in \F_q[x]$ is an $S$-translation invariant polynomial if $f(x+s)=f(x)$ for all $s\in S$. 

\begin{example}
$f(x)=x^q-x$ is an $\F_q$-translation invariant polynomial.
\end{example}

Let $C_S(n)$ denote the set of all $S$-translation invariant monic irreducible polynomials of degree $n$ over $\F_q$.
Suppose that $a, b\in \F_q^{*}$ and $f(x)\in \F_q[x]$ is a polynomial such that $f(x+a)=f(x)$ and $f(x+b)=f(x)$. Clearly $f(x+0)=f(x)$ and then, by induction, we have that 
$$f(x+ia+jb)=f(x)$$ for all $i, j\in \F_p$. In particular, if $S'\subset \F_q$ is the $\F_p-$vector space generated by $S$, then $C_{S'}(n)\supset C_S(n)$. The converse inclusion is obviously true and so $C_{S'}(n)=C_S(n)$.
From now, $S\subset \F_q$ will denote an $\F_p-$vector space of dimension $r>0$.  

\begin{remark}\label{obs} If $f(x)\in \F_q[x]$ is an $S$-translation invariant polynomial and $a\in \F_q^*$, then $g(x)=a^{n}f(a^{-1}x)$ is an $S'$-translation invariant polynomial, where $S'=\{as | s\in S\}$. Moreover, if $S$ is an $\F_p-$vector space then so is $S'$.
\end{remark}
The observations above lead us to the following definition:

\begin{definition}
Let $S, S'\subset\F_q$ be $\F_p$-vector spaces. We say that $S$ and $S'$ are $\F_q-$\textit{linearly equivalent} and write $S\sim_{\F_q} S'$ if there exists $a\in \F_q^*$ such that $S'=\{as| s\in S\}$.
\end{definition}

It is immediate from definition that the relation $\sim_{\F_q}$ is an equivalence relation. Moreover, this relation gives an interesting invariant:

\begin{lemma}\label{igual}
If $S\sim_{\F_q} S'$, then $|C_S(n)|=|C_{S'}(n)|$ for all $n\in \N$.
\end{lemma}

\begin{proof}

If $S\sim_{\F_q}S'$, then $S'=\{as|s\in S\}$ for some $a\in \F_q^*$. Let $I_n(q)$ be the set of all monic irreducible polynomials of degree $n$ over $\F_q$ and 

$$\begin{matrix}
\tau_{a}:&I_n(q)&\to& I_n(q)\\
&f(x)&\mapsto &a^nf(a^{-1}x).
\end{matrix}$$

Clearly $\tau_a$ is well defined and is one to one. From Remark \ref{obs} we know that if $f(x)\in C_S(n)$ then $\tau_{a}(f(x))\in C_{S'}(n)$.  Therefore we have that $|C_S(n)|\le |C_{S'}(n)|$. Since the relation $\sim_{\F_q}$ is symmetric, it follows that $|C_{S'}(n)|\le |C_{S}(n)|$ and this completes the proof.
\end{proof}

The following theorem gives a characterization of the $S$-translation invariant polynomials over the finite field $\F_q$.

\begin{theorem}\label{composicao}
Let $S\subset \F_q$ be an $\F_p-$vector space of dimension $r>0$ and $$P_S(x): =\prod_{s\in S}(x-s).$$ Then $g(x)\in \F_q[x]$ is an $S$-translation invariant polynomial if, and only if, there exists some polynomial $f(x)\in \F_q[x]$ such that 
$g(x)=f(P_S(x))$. In particular, any $S$-translation invariant polynomial has degree divisible by $\deg P_S(x)=|S|=p^r$.
\end{theorem}
\begin{proof}
Assume that $g(x)$ is an $S$-translation invariant polynomial over $\F_q$. We proceed by induction on $n=\deg g(x)$. If $g(x)$ is constant then there is nothing to prove. Suppose that the statement is true for all polynomials of degree at most $k$ and let $g(x)$ be an $S$-translation invariant polynomial of degree $k+1$. We have $g(0)=g(s)$ for all $s\in S$ and so the polynomial $g(x)-g(0)$ has degree $k+1>0$ and vanishes at $s$ for all $s\in S$. In particular we have that 
\begin{align}\label{um} g(x)-g(0)=P_S(x) G(x)\end{align}
 for some non-zero polynomial $G(x)\in \F_q[x]$. Since $g(x+s)-g(0)=g(x)-g(0)$ and $P_S(x+s)=P_S(x)$ for all $s\in S$, by equation \eqref{um} it follows that $G(x)$ is an $S$-translation invariant polynomial over $\F_q$ and $\deg G(x)< \deg g(x)$. By the induction hypothesis we have that $G(x)=F(P_S(x))$ for some $F(x)\in \F_q[x]$. Therefore $g(x)=P_S(x) F(P_S(x))+g(0)$ and so $g(x)=f(P_S(x))$ where $f(x)=xF(x)+g(0)$. The converse is obviously true.
\end{proof}

If $1, a\in S$ where $a\not\in \F_p$, we have the following:
\begin{corollary}\label{cor}
Let $S\subset \F_q$ be an $\F_p-$vector space such that  $1, a\in S$ where $a\not\in \F_p$. Then any $S$-translation invariant polynomial in $\F_q[x]$ is of the form
$$f(x^{p^2}-x^p(1+(a-a^p)^{p-1})+x(a-a^p)^{p-1})$$
for some polynomial $f(x)\in \F_q[x]$.
\end{corollary}
\begin{proof}
Since $1, a\in S$ and $a\not\in \F_p$, we have that any $S$-translation invariant polynomial is also a $S'$-translation invariant where $S'=\langle 1, a\rangle_{\F_p}=\{j+ai| i, j\in\F_p\}$. In addition, notice that 
$$P_{S'}(x)=\prod_{0\le i, j\le p-1}(x-j-ai)=x^{p^2}-x^p(1+(a-a^p)^{p-1})+x(a-a^p)^{p-1}.$$ 
\end{proof}

The main result of this paper is the following:

\begin{theorem}\label{main}
Let $\F_q$ be the finite field with $q=p^k$ elements and $S\in \F_q$ be an $\F_p-$vector space of dimension $r>0$.
\begin{enumerate}[a)]
\item If $r>1$ then $|C_S(n)|=0$ for all $n$.
\item If $r=1$ and $n$ is not divisible by $p$ then $|C_S(n)|=0$.
\item If $r=1$ and $n=pm$, then $$|C_S(n)|=\frac{p-1}{pm}\sum_{d|m\atop \gcd(d, p)=1}q^{m/d}\mu(d).$$
\end{enumerate}
\end{theorem}

\subsection{Lemmata}
In order to prove Theorem \ref{main} we have to discuss the irreducibility of the polynomials $f(P_S(x))\in \F_q[x]$. In this direction, Lemmas \ref{agou} and \ref{russian} give some criteria to ensure the irreducibility of special polynomial compositions and Lemma \ref{count} will be useful in the proof of the enumeration formula in Theorem \ref{main}.

\begin{lemma}[\cite{agou}]\label{agou} 
Let $f(x)=x^{n}+Bx^{n-1}+\cdots+c\in \F_q[x]$ be an irreducible polynomial, where $q=p^k$ is a prime power. The polynomial $f(x^{p^{2t}}-ax^{p^t}-bx)$ is also irreducible over $\F_q[x]$  if and only if the following conditions are satisfied:

\begin{enumerate}[a)]
\item $p=2, t=1, n$ is odd and $B\ne 0$,

\item $\gcd(x^3-ax-b, x^{2^k}-x)\ne 1$,

\item $\rm{{Tr}}$$_{L/K}(\beta^{-2}B)=\Tr$$_{L/K}(\alpha^{-2}\beta)=1$, where $L=\F_{2^k}, K=\F_2$ and $\alpha, \beta$ are two elements in $L$ such that $\alpha^2+\beta=a$ and $\alpha\beta=b$.

\end{enumerate}
\end{lemma}

\begin{lemma}[\cite{LiNi}, Theorem 3.82]\label{russian}
Let $f(x)=x^{n}+a_{n-1}x^{n-1}+\cdots+c\in \F_q[x]$ be an irreducible polynomial, $p=\Char {\F_q}$ and $b\in \F_q$. The polynomial $f(x^p-x-b)$ is also irreducible over $\F_q[x]$  if and only if $\Tr$$_{L/K}(nb-a_{n-1})\ne 0$ where $L=\F_q$ and $K=\F_p$.
\end{lemma}
 
\begin{lemma}[\cite{count}]\label{count}
The number of monic irreducible polynomials over $\F_q$ with degree $n$ and a given trace $a\ne 0$ is
$$\frac{1}{qn}\sum_{d|n\atop \gcd(d, p)=1}q^{n/d}\mu(d).$$
\end{lemma}

\subsection{Proof of Theorem \ref{main}}
If $S\subset \F_q$ is any $\F_p-$vector space of dimension $r>0$ and $a\in S\setminus \{0\}$, then $a^{-1}S\sim_{\F_q}S$ and $1\in a^{-1}S$. Thus, using Lemma \ref{igual}, we can suppose that $1\in S$.

\begin{enumerate}[a)]
\item Let $n$ be a positive integer and $S\subset \F_q$ be an $\F_p-$vector space of dimension $r>1$ such that $1\in S$. Since $r>1$, there is some element $\gamma\in S\setminus \F_p$. It follows from Corollary \ref{cor} that any $g(x)\in C_S(n)$ is of the form 
$$f(x^{p^2}-x^p(1+(\gamma-\gamma^p)^{p-1})+x(\gamma-\gamma^p)^{p-1}).$$
Therefore, by Lemma \ref{agou}, $g(x)$ is reducible whenever $p>2$. 

If $p=2$, we have that $g(x)=f(x^4-ax^2-bx)$ where $a=\gamma^2+\gamma+1$ and $b=\gamma^2+\gamma$. Now consider the following system of equations:

\begin{equation}\label{two}
\begin{cases} \alpha^2+\beta&=  a \\ \alpha\beta& = b.\end{cases}
\end{equation}
Notice that, for any solution $(\alpha, \beta)$ of system \eqref{two}, $\alpha$ is a root of the equation $y^3-ay+b=0$ and $\beta=a-\alpha^2$. In particular the system \eqref{two} has at most three solutions and since $(\alpha, \beta)=(1, \gamma^2+\gamma), (\gamma, \gamma+1), (\gamma+1, \gamma)$ are solutions these are all of them.

Condition c) of Lemma \ref{agou} says that, for some solution $(\alpha, \beta)$ of \eqref{two}, we have $\Tr_{K/L}(\alpha^{-2}\beta)=1$ where $L=\F_{2^k}$ and $K=\F_2$. In particular, if $(\alpha, \beta)=(1, \gamma^2+\gamma)$ then
$$\Tr_{K/L}(\alpha^{-2}\beta)=\Tr_{K/L}(\gamma^2)+\Tr_{K/L}(\gamma)=\Tr_{K/L}(\gamma)+\Tr_{K/L}(\gamma)=0\ne 1.$$
 The other two solutions of \eqref{two} satisfy $\beta=\alpha+1$ and in these cases we have
$$\Tr_{K/L}(\alpha^{-2}\beta)=\Tr_{K/L}(\alpha^{-1})+\Tr_{K/L}((\alpha^{-1})^2)=\Tr_{K/L}(\alpha^{-1})+\Tr_{K/L}(\alpha^{-1})=0\ne 1.$$
Thus $g(x)$ is never irreducible if $r>1$ and this completes the proof of a).
\\

\item This follows directly from Theorem \ref{composicao} since any polynomial of the form $f(P_S(x))$ has degree divisible by $|S|=p$.
\\

\item Since $r=1$ and $1\in S$, we have $S=\F_p$. Let $m$ be any positive integer and $g(x)$ be an irreducible polynomial of degree $pm$. From Lemma \ref{composicao} we have that $g(x)\in C_S(pm)$ if, and only if, $g(x)$ is of the form $f(P_S(x))=f(x^p-x)$ where $f(x)=x^m+a_{m-1}x^{m-1}+\cdots+a_0$ is an irreducible polynomial over $\F_q$. 

From Lemma \ref{russian}, the polynomial $f(x^p-x)\in \F_q[x]$ is irreducible if, and only if, $\Tr$$_{L/K}(a_{m-1})\ne 0$ where $L=\F_q$ and $K=\F_p$. It is well known that for each $a\in K$, the equation $\Tr$$_{L/K}(x)=a$ has $\frac{q}{p}$ distinct solutions. Thus for exactly $(p-1)\frac{q}{p}$ values of $a_{m-1}\in \F_q$ we have $\Tr$$_{L/K}(-a_{m-1})\ne 0$ and any of these elements is nonzero. For a fixed $\beta\in \F_q^*$, by Lemma \ref{count}, there exist exactly 
$$\frac{1}{qm}\sum_{d|m\atop \gcd(d, p)=1}q^{m/d}\mu(d)$$
monic irreducible polynomials over $\F_q$ with degree $m$ and trace $\beta$. Thus the number of monic irreducible polynomials $f(x)\in \F_q[x]$ of degree $m$ such that $f(x^p-x)\in \F_q[x]$ is also irreducible is equal to
$$(p-1)\frac{q}{p}\cdot \frac{1}{qm}\sum_{d|m\atop \gcd(d, p)=1}q^{m/d}\mu(d)=\frac{p-1}{pm}\sum_{d|m\atop \gcd(d, p)=1}q^{m/d}\mu(d).$$
Since the polynomials $f(x^p-x)$ are all distinct when $f(x)$ runs through $I_q(m)$, we are done.
\end{enumerate}

\section{Miscellanea}
Using the results obtained in Section 2, we give alternative proofs of two interesting results concerning the action of $\GL_2(\F_q)$ on irreducible polynomials.

\subsection{Polynomials invariant under homotheties}

Given $a\in \F_q\setminus\{0, 1\}$, we are interested in counting the number of monic irreducible polynomials $g(x)$ that are  invariant under the homothety $x\mapsto ax$, i.e., $g(ax)=g(x)$. We have the following characterization:
\begin{theorem}\label{one}
Let $f(x)$ be a polynomial over $\F_q$. Then $f(ax)=f(x)$ if and only if there exists $g(x)\in \F_q[x]$ such that 
 $$f(x)=g(P_a(x)),$$ where $P_a(x)=\prod_{i=0}^{k-1}(x-a^i)=x^k-1$. In particular any polynomial invariant under the homothety $x\mapsto ax$ has degree divisible by $k=\ord(a)$.
\end{theorem}
\begin{proof}
Notice that if $f(ax)=f(x)$ and $k=\ord(a)$ is the order of $a\in \F_q$, then $f(1)=f(a^i)$ for all $0\le i\le k-1$; from now, the proof is quite similar to the one of Theorem \ref{composicao}.
\end{proof}

Let $N_a(nk)$ be the number of monic irreducible polynomials $f(x)\in \F_q[x]$ of degree $nk$ such that $f(ax)=f(x)$, $k=\ord(a)$. Also, let $L(n, k)$ be the number of monic irreducible polynomials of the form $F(x^k)$, where $F$ has degree $n$. Notice that, from Theorem \ref{one}, $N_a(nk)$ is exactly the number of monic irreducible polynomials of the form $f(x^k-1)$ where $f$ has degree $n$. Since $f(x^k-1)=F(x^k)$ for $F(x)=f(x-1)$ and $\{f(x-1) | f\in I_q(n)\}=I_q(n)$, the number of monic irreducible polynomials $F$ of degree $n$ for which $F(x^k)$ is also irreducible is the same for the composition $F(x^k-1)$. In particular we have proved that $N_a(nk)=L(n, k)$. Combining the previous equality with the enumeration formula for $L(n, k)$ presented in [\cite{cohen}, Theorem 3] we directly deduce the enumeration formula of Garefalakis for homotheties:

\begin{theorem}[\cite{gfq}, Theorem 4]
If $n$ is not divisible by $k=\ord(a)$, then $N_a(n)=0$ and, if $n=mk$, we have
$$N_a(mk)=\frac{\varphi(k)}{mk}\sum_{d|m\atop \gcd(d, k)=1}\mu(d)(q^{m/d}-1).$$
\end{theorem}

\subsection{The action of $\rm{PGL}$$_2(\F_q)$ on irreducible polynomials}
Consider the projective linear group $G=\rm{PGL}$$_2(\F_q)\approx \GL_2(\F_q)/\sim$, where $A\sim B$ if $A=\lambda B$ for some $\lambda\in \F_q^*$. We are interested to find the monic irreducible polynomials $f(x)\in \F_q[x]$ such that $A\circ f=f$ for any $A\in G$. As it was shown in [\cite{general}, Proposition 4.8] in general there are no such polynomials:

\begin{theorem}
Let $f\in \F_q[x]$ be a monic irreducible polynomial of degree $n\ge 2$. Suppose that $A\circ f=f$ for all $A\in \rm{PGL}$$_2(\F_q)$. Then $n=q=2$ and $f(x)=x^2+x+1\in \F_2[x]$.
\end{theorem}

Here we give an alternative proof of this fact: let $f\in \F_q[x]$ be as above. Since all $2\times 2$ matrices corresponding to translations belong to $\rm{PGL}$$_2(\F_q)$, $f$ must be an $\F_q-$translation invariant and, from Theorem \ref{main} part a), we know that this is only possible when $\F_q$ is a prime field. Thus $q=p$ prime. Also, from Theorem \ref{composicao}, we know that there exists some polynomial $g(x)$ with nonzero trace such that $f(x)=g(x^p-x)$.

Now, if $p>2$, then there is some element $a\in \F_p\setminus \{0, 1\}$ such that $k=\ord(a)>1$. Since all homotheties also belong to $\rm{PGL}$$_2(\F_q)$, it follows that $f(x)=f(ax)$ or, equivalently, $g(x^p-x)=g_2(x^p-x)$, where $g_2(x)=g(ax)$. Therefore $g(x)=g_2(x)$, i.e., $g(x)$ is also invariant under the homothety $x\mapsto ax$. From Theorem \ref{one} we have that $g(x)=h(x^k-1)$ for some $h(x)\in \F_p[x]$. Since $k>1$, a direct calculation shows that $g(x)$ has trace equal to zero, a contradiction. 

Thus $p=2$, $f(x)=g(x^2+x)$ and $\deg f(x)=2N=2\deg g(x)$. Let $$\gamma, \gamma^2, \cdots, \gamma^{2^{2N-1}}$$ be the roots of $f(x)$. Notice that $f(\gamma+1)=g(\gamma^2+\gamma)=0$ and then $\gamma+1=\gamma^{2^j}$ for some $0< j<2N$. In other words, $F(x)=x^{2^j}+x+1\in \F_2[x]$ has $\gamma$ as a root. By hypothesis, $f(x)\in \F_2[x]$ is irreducible, and thus $f(x)$ divides $F(x)$. Now, since the inversion $f^*(x)=x^nf(1/x)$ also belong to $\rm{PGL}$$_2(\F_q)$, it follows that $x^nf(1/x)=f(x)$ and so $f^*(x)=f(x)$ divides $F^*(x)=x^{2^j}+x^{2^j-1}+1$. Thus $f(x)$ divides $F(x)+x(F^*(x)+F(x))=x^2+x+1$. Since $\deg f(x)\ge 2$, the only possibility is $f(x)=x^2+x+1$. It can be easily  verified that $f(x)=x^2+x+1\in \F_2[x]$ is irreducible  and $A\circ f=f$ for all $A\in \rm{PGL}_2$$(\F_2)$.

\section{$p-$ subgroups of $\GL_2(\F_q)$}
Let $p=\Char \F_q$. If $S\subset \F_q$ is an $\F_p-$vector space of dimension $r$, then the set of translations $\{x+s; s\in S\}$ corresponds to the $p-$group $H_{S}$ of the matrices $\left(\begin{matrix}
1&s\\
0&1
\end{matrix}\right)$ where $s$ runs through $S$, $|H_S|=|S|=p^r$. This correspondence is an one-to-one correspondence between the subgroups of $H_{\F_q}$ and $\F_p-$ vector spaces of $\F_q$.

For a $p-$group $H \subset \GL_2(\F_q)$, $I_H(n)$ denotes the set of all monic irreducible polynomials in $\F_q$ of degree $n$ such that $A\circ f=f$ for all $A\in H$. In the previous correspondence, the sets $I_{H_S}(n)$ and $C_S(n)$ (defined in Section 2) are the same.

Combining the observations above with Theorem \ref{main} we are able to characterize the numbers $|I_H(n)|$:

\begin{theorem}
Let $H\subset \GL_2(\F_q)$ be any group of order $p^r$. 
\begin{enumerate}[a)]
\item If $r>1$ or $n$ is not divisible by $p$, then $|I_H(n)|=0$.
\item If $r=1$ and $n=pm$, then 
$$|I_H(n)|=\frac{p-1}{pm}\sum_{d|m\atop \gcd(d, p)=1}q^{m/d}\mu(d).$$
\end{enumerate}
\end{theorem}
\begin{proof}
Notice that $|\GL_2(\F_q)|=q(q-1)^2(q+1)$ and the sets of all translations in $\F_q$ corresponds to a Sylow $p-$ subgroup $G$ of $\GL_2(\F_q)$. Let $H\subset \GL_2(\F_q)$ be any group of order $p^r$. We know that $H$ is contained in some Sylow $p-subgroup$ $K$ of $\GL_2(\F_q)$. Since all Sylow $p-$subgroups are conjugate, there is some $A\in \GL_2(\F_q)$ such that $A^{-1}KA=G$ and then $A^{-1}HA$ is a subgroup of $G$ and has order $p^r$.
Consider now the following map: 
$$\begin{matrix}
\tau_{H, A}:&I_H(n)&\to& I_{A^{-1}HA}(n)\\
&f(x)&\mapsto &k_{A, f, n}(A^{-1}\circ f(x)),
\end{matrix}$$
where $k_{A, f, n}$ is the only element in $\F_q$ such that $k_{A, f, n}(A^{-1}\circ f(x))$ is monic. A direct calculation shows that $\tau_{H, A}$ is well defined. Now, suppose that $\tau_{H, A}(f)=\tau_{H, A}(g)$ for some $f, g\in I_H(n)$, i.e., $k_{A, f, n}(A^{-1}\circ f(x))=k_{A, g, n}(A^{-1}\circ g(x))$. Applying $A$ we get $k_{A, f, n}f=k_{A, g, n}g$. Since $f$ and $g$ are monic irreducible we conclude that $f=g$ and then $\tau_{H, A}$ is one to one. In a similar way we can define a map from $I_{A^{-1}HA}(n)$ into $I_H(n)$. Thus $|I_{A^{-1}HA}(n)|=|I_H(n)|$.

The advantage is that $A^{-1}HA$ is a $p-$subgroup of $G$, i.e., $I_{A^{-1}HA}(n)$ is equal to $C_S(n)$ for some $\F_p-$vector space $S$ of dimension $r$. Now, the result follows from Theorem \ref{main}.
\end{proof}

\begin{center}\textbf{Acknowledgments} \end{center}
I would like to thank Daniel Panario for some helpful suggestions in the conclusion of this work and John MacQuarrie for the English corrections.


\end{document}